\documentclass[letterpaper, 10 pt, conference]{ieeeconf}  % Comment this line out
\IEEEoverridecommandlockouts                              % This command is only
\usepackage{tikz}
\usepackage{amsmath}
\usepackage{latexsym, amssymb, amsmath, amsbsy,amsopn, amstext}
\usepackage{graphicx,hyperref,booktabs}
\usepackage{cancel, color}
\usepackage{epstopdf,subcaption}
\graphicspath{{images/}}
\newtheorem{theorem}{Theorem}[section]
\newtheorem{lemma}[theorem]{Lemma}
\newtheorem{corollary}[theorem]{Corollary}
\newtheorem{proposition}[theorem]{Proposition}
\newtheorem{conjecture}[theorem]{Conjecture}
\newtheorem{definition}[theorem]{Definition}
\newtheorem{example}[theorem]{Example}
\newtheorem{remark}[theorem]{Remark}

\DeclareMathOperator{\Span}{Span}

\DeclareMathOperator{\Int}{Int}

\DeclareMathOperator{\Con}{Con}
\DeclareMathOperator{\Lie}{Lie}
\DeclareMathOperator{\Const}{Const}

\def\A{\mathcal{A}}

\def\E{\mathcal{E}}
\def\I{\mathcal{I}}
\def\R{\mathbb{R}}

\def\Z{\mathbb{Z}}

\def\X{\mathfrak{X}}
\def\dd{\mathrm{d}}
\def\ad{\mathrm{ad}}
\def\T{\mathrm{T}}

\def\vphi{\varphi}
\def\vd{\varDelta}

\def\ra{\rightarrow}

\def\leqs{\leqslant}
\def\geqs{\geqslant}

\title{\LARGE \bf
	Global Controllability Criteria and Motion Planning of Regular Affine Systems With Drifts
}

\author{Zhengping Ji, Xiao Zhang and Daizhan Cheng% <-this % stops a space
	\thanks{This work is supported partly by NNSF 62073315 of China, and China Postdoctoral Science Foundation 2021M703423 and 2022T150686.}% <-this % stops a space
\thanks{D. Cheng is with the Academy of Mathematics and Systems Science, Chinese Academy of Sciences, Beijing 100190, P.R.China, {\tt\small dcheng@iss.ac.cn}}
\thanks{X. Zhang is with the National Center for Mathematics and Interdisciplinary Sciences \& the Key Laboratory of Systems and Control, Academy of Mathematics and Systems Science, Chinese Academy of Sciences, Beijing 100190, P.R.China, {\tt\small xiaozhang@amss.ac.cn}}
\thanks{Z. Ji is with the Key Laboratory of Systems and Control, Academy of Mathematics and Systems Science \& School of Mathematical Sciences, University of Chinese Academy of Sciences, Beijing 100190, P.R.China, {\tt\small jizhengping@amss.ac.cn} }
}

\begin{document}

	\maketitle
	\thispagestyle{empty}
	\pagestyle{empty}

\begin{abstract}
	
In this article, we give a condition for the global controllability of affine nonlinear control systems with drifts on Euclidean spaces. Under regularity assumptions, the condition is necessary and sufficient in the codimension-1 and codimension-2 cases, and holds for systems of higher codimensions under mild restrictions. We then investigate motion planning problems for codimension-1 affine systems, and give proof of the global existence of the lift to control curves for certain drifted systems using the homotopy continuation method. 

\end{abstract}

\section{Introduction}\label{S1}

Finding necessary and sufficient conditions for the global controllability of affine control systems
\begin{align}\label{0.1}
	\dot{x}(t)=f|_{x(t)}+\sum_{i=1}^{m}u^i(t)g_i|_{x(t)}
\end{align}
with  smooth vector fields $f$, $\{g_i\}_{i=1}^m$ and measurable, essentially bounded controls $\{u^i\}_{i=1}^m$ is a basic problem in control theory. It has been investigated since the 1930s when Chow and Rachevskii provided bracket-generating conditions \cite{agr20} for the global controllability of the driftless version of (\ref{0.1}) as 
\begin{align}\label{1.1}
	\dot{x}(t)=\sum_{i=1}^{m}u^i(t)g_i|_{x(t)},
\end{align}
where $\{u^i\}_{i=1}^m$, $\{g_i\}_{i=1}^m$ are as in (\ref{0.1}). The global controllability condition of (\ref{0.1}) has been solved in some simple cases, such as when the drift vector field $f$ lies in the Lie algebra generated by $\{g_i\}_{i=1}^m$ \cite{agr04}, for planar systems \cite{su07}, and for codimension-$1$ systems under assumptions requiring global generating families \cite{hu82}. The general case, however, remains open. 

Suppose the system (\ref{0.1}) is defined on a manifold $M$. For $\forall p\in M$, the reachable set of $p$ is defined as
$$
\A(p):=\bigcup_{T\geqslant0}\A(p,T),
$$
where 
{\small $$
	\A(p,T):=\Bigg\{x(T)\Bigg|
	\begin{array}{l}
		\exists u^i\in L^1([0,T],\R), i=1,\cdots,m, \\
		(x(t),u(t)) ~\text{satisfying (\ref{0.1})},~x(0)=p
	\end{array}\Bigg\}.
	$$}
When $\A(p)$ contains an open set with respect to the topology of $M$, we say that the system (\ref{0.1}) is strongly attainable at $p$. If for $\forall p\in M$, $\A(p)=M$, we call (\ref{0.1}) globally controllable.
%	When $\A(x)=M$, we say that (\ref{0.1}) is globally controllable at $x$; 

For globally controllable driftless affine systems, the motion planning problem has been widely considered. Its basic objective is: given any pair $(p,q)$ in the state space of a control system, to design a control that yields an admissible trajectory steering $p$ to $q$. Several methods have been proposed for this problem, such as the nilpotent approximation \cite{ch13}, the loop method \cite{son}, and steering with sinusoidal controllers \cite{mu}. However, there have been few investigations of motion planning problems for drifted affine systems due to the difficulty of dealing with drift and the lack of controllability conditions \cite{zu}.

Over the past thirty years, the homotopy continuation method (HCM) introduced by \cite{su92} has been applied to motion planning of affine systems and has shown good performance in numerical practice \cite{al}. The basic idea of HCM for driftless affine systems is to lift a curve in $M$ to one in the control space ${\cal U}$ through the endpoint map $\E_p:{\cal U}\ra M$, $u\mapsto\gamma_{u,p}(1)$, where $\gamma_{u,p}$ is the solution of the control system (\ref{1.1}) corresponding to $u$ starting from $p\in M$. If a path $\gamma(t)$ connecting $q_0$ and $q$ on $M$ can be lifted through $\E_p$ to a path $\zeta$ on $\cal U$ starting from $u_0$, satisfying $\E_p(u_0)=q_0$, that is, 
$$
\exists\zeta:[0,1]\ra {\cal U}, ~{\text{s.t.}}~\gamma(t)=\E_p(\zeta(t)), ~\forall t\in[0,1],
$$
then the trajectory $\zeta(1):=(u_1^*(t),\cdots,u_m^*(t))\in \cal U$ will give the required controls $u_1^*(t),\cdots,u_m^*(t)$ driving system (\ref{0.1}) from $p$ to $q$ by time $t=1$. 

A sufficient condition for the curve $\gamma(t)$ to be lifted is the existence of the solution to the following path-lifting equation (PLE):
\begin{align}\label{0.3}
	D\E_p|_{\zeta(t)}\frac{\dd \zeta(t)}{\dd t}=\frac{\dd\gamma(t)}{\dd t}.
\end{align}
If (\ref{0.3}) admits a global solution on $[0,1]$, then the motion planning problem is solved by $u(t)=\zeta(1)$ (numerically, (\ref{0.3}) can be solved by a finite-dimensional approximation of the control space \cite{al}). Thus, the problem is reduced to avoiding singular points of $D\E_p$ and finding conditions on the vector fields of (\ref{1.1}) for the PLE to be solved globally on $[0,1]$. It was proved \cite{ch06} that if the Moore-Penrose inverse of $D\E_p$ at $u$, denoted by $P(u)$, has linear growth of $\|u\|$ on any compact set, then (\ref{0.3}) is globally solved on $[0,1]$. Such conditions are satisfied for some special types of driftless systems \cite{ch06,su93}, but so far there are no known results on the application of HCM to drifted systems. {This is because the domain of the endpoint map should be changed to a product space for drifted cases, and due to the existence of the drift, the time parameter cannot be restricted to $[0,1]$ as in the driftless case, which will also change the PLE consequently.}

In this article, we consider affine systems defined on $\R^n$ with nonzero drifts. We will show that when the control Lie algebras of the corresponding driftless systems are regular and {the controls are $L^2$}, the sufficient conditions given in \cite[Theorem 6.1]{che10} are also necessary for global controllability of systems of codimension $1$, and of systems of codimension $2$ or higher under certain restrictions. We then discuss the motion planning problems of globally controllable affine systems. When a system of codimension $1$ allows globally generating vector fields, we provide a simple steering algorithm; and for general drifted systems, we adopt the HCM and show the global existence of the lifting curve when the vector fields satisfy certain restrictions.

\section{Necessary and Sufficient Conditions for Controllability: Codimension $1$}\label{S2}

Consider the system (\ref{0.1}) defined on $\R^n$. From now on we denote by $G:=\Lie\{g_i\}$ the Lie algebra generated by $\{g_i\}_{i=1}^m$, i.e. the smallest Lie algebra containing $\{g_i\}_{i=1}^m$.	We make the following assumption:
\begin{enumerate}
	\item[\bf (A1)] $G$ is regular, i.e. $\dim G|_x=\Const$, $\forall x\in \R^n$.
\end{enumerate}

Let $M$ be a manifold and $\vd$ an involutive distribution on $M$ (in this article we assume that all distributions are smooth). Given $x\in M$, denote by $\I_\vd(x)$ the maximal integral submanifold passing through $x$ corresponding to $\vd$. Under the assumption {\bf (A1)}, we define the codimension of the system (\ref{0.1}) as $n-\dim G$, and if the system is of codimension $k$, then every integral manifold $\I_{G}(x)$ is an $(n-k)$-dimensional injectively immersed edgeless submanifold in $\R^n$, giving $\R^n$ a codimension-$k$ foliation structure \cite{ca00}, where every maximal integral submanifold is a leaf. In fact, $\I_G(x)$ is the reachable set of $x$ corresponding to the driftless system (\ref{1.1}). This is the classical Chow-Rashevskii theorem \cite{agr20}.

Given a manifold $M$ and a vector field $X\in\X(M)$, denote by $\psi^X_t(x)$ the flow of $X$ starting from $x\in M$ after time $t$. The critical question in the problem of controllability of affine nonlinear system (\ref{0.1}) is whether $\{\psi_{t}^{f}(x)\}_{t<0}$ can intersect the reachable set of $x$ (we assume that all vector fields mentioned are complete).

\begin{lemma}\label{l1.11}
	Consider the control system (\ref{0.1}) of codimension $1$. Denote by $\A_f(x)$ the reachable set of a state $x\in M$ with respect to (\ref{0.1}). Consider the following system 
	\begin{align}\label{0.2}
		\dot{x}(t)=-f|_{x(t)}+\sum_{i=1}^{m}u^i(t)g_i|_{x(t)}
	\end{align}
	where $f$, $\{g_i\}$ are vector fields defined as in (\ref{0.1}). Denote by $\A_{-f}(x)$ the reachable set of $x\in M$ with respect to (\ref{0.2}). Then the following conditions are equivalent:
	\begin{enumerate}
		\item The system (\ref{0.1}) is globally controllable;
		\item The system (\ref{0.2}) is globally controllable;
		\item $\A_f(p){\cap}\A_{-f}(p)\neq\varnothing$, $\forall p\in M$.
	\end{enumerate}
\end{lemma}
\begin{proof}
	$\forall p,q\in M$, $q\in\A_{-f}(p)$ is equivalent to $p\in\A_f(q)$, hence $\A_{-f}(p)=M$, $\forall p\in M$ is equivalent to $\A_f(p)=M$, $\forall x\in M$.
	
	On the other hand, if $\A_f(p){\cap}\A_{-f}(p)\neq\varnothing$, then there exists a control $u:[0,1]\ra \R$ s.t. the trajectory $\gamma_u:[0,T]\ra M$ corresponding to $u$ satisfies $\gamma_u(0)=\gamma_u(T)=p$. Choose a neighbourhood $U$ of $p$ such that there exists a point $y\in\gamma_u([0,T])$ satisfying $q\notin \A_f(p)\cap U$; then there is a neighbourhood $V$ of $p$ such that $V\subset\A_f(q)$. Also note that $A_f(q)\subset\A_f(p)$, one can see that $A_f(p)$ contains an open neighborhood of $p$, hence {by the well known fact that pointwise local controllability implies global controllability \cite{gra}}, the system is globally controllable.		
\end{proof}

\begin{proposition}\label{p2.1}
	Consider system (\ref{0.1}) of codimension $1$. {Suppose $\exists p\in M$ such that there is a open neighbourhood $U_p$ of $p\in M$ such that $\I_G(p)$ is dense in $U_p$, then the system is globally controllable from any point $q_1\in U_p$ to another point $q_2\in U_p$, if and only if, the system is bracket-generating} at each point in $U_x$.
\end{proposition}
\begin{proof}
	We need to show that $\A_f(q){\cap}\A_{-f}(q)\neq\varnothing$, $\forall q\in U_p$. Since $\I_G(p)$ is dense in an open set of $M$, $\A_f(q)$ and $\A_{-f}(q)$ are both dense in $U_p$; also they both have nonempty interior {with respect to the topology of $M$}, so there are two open and dense subsets in $U_p$, and they must intersect. The conclusion follows by Lemma \ref{l1.11}.
\end{proof}

Therefore, {for the case where (\ref{0.1}) is defined over $\R^n$}, we make another assumption:
\begin{enumerate}
	\item[\bf (A2)]\label{a2} {For any open set $U\subset\R^n$, $\I_G(x)$ is not dense in $U$.}
\end{enumerate}
This is equivalent to say that $\I_G(x)$ is an embedding in $\R^n$, $\forall x\in \R^n$ \cite[Lemma 3.2]{ro17}.

The following lemma is crucial for our proof of the controllability criteria.

\begin{lemma}\label{l1.2}
	A maximal integral manifold corresponding to a codimension-$1$ involutive distribution satisfying {\bf (A2)} separates $\R^n$ into at least two connected components.
\end{lemma}

\begin{proof}
	According to \cite{fe88}, consider an immersion $f: M\ra N$, where $M$ and $N$ are edgeless manifolds of dimensions $n-1$ and $n$, respectively, if the inverse images of compact sets are compact and $H_1(N;\Z/2\Z)=0$, then $N\backslash f(M)$ is not connected. In our context, since the codimension $1$ integral manifolds satisfying {\bf (A2)} are embeddings and $H_1(\R^n;\Z/2\Z)=0$, they satisfy these conditions and hence $\R^n\backslash\I_G(x)$ is not connected.
\end{proof}

We call a system (\ref{0.1}) regular if its corresponding Lie algebra $G$ satisfies both {\bf (A1)} and {\bf (A2)}. The key argument of the controllability criteria is based on the so-called shift of the vector fields of regular systems, defined as follows.

\begin{definition}[\cite{che10}]\label{d1.1}
	Let $M$ be a smooth manifold and $\vd$ a distribution on $M$. Given a vector field $X\in\X(M)$, by $X\in\vd$ we mean that $X|_x\in\vd|_x$, $\forall x\in M$.
	\begin{enumerate}
		\item Given $x,y\in M$ and tangent vectors $\xi\in T_xM$, $\eta\in T_yM$, if  $\exists g_1,\cdots,g_m\in\vd$, $t_1,\cdots,t_m\geqslant0$, s.t.
		\begin{align}\label{1.2}
			x=\vphi_{t_m}^{g_m}\cdots\vphi_{t_1}^{g_1}(y),~
			\xi=(\vphi_{t_m}^{g_m})_*\cdots(\vphi_{t_1}^{g_1})_*(\eta),
		\end{align}
		then $\xi$ is called the $\vd$-shift of $\eta$ from $y$, denoted by $\xi=\vd_*(y,x)(\eta)$.
		\item Given a subset $S\subset \X(M)$, the $\vd$-shift of $S$ to a point $x$ is defined as 
		\begin{align}\label{1.3}
			\vd_*S(x):=\{\vd_*(y,x)(S|_y)|y\in \I_{\vd}(x)\}\subset T_xM.
		\end{align}
		\item A vector field $X$ is called $\vd$-invariant, if $\vd_*X(x)=X|_x$, $\forall x\in M$. This is equivalent to that $[X,Y]=0$, $\forall Y\in\vd$.
	\end{enumerate}

\end{definition}

With the above notions, now we state the main theorem for codimension-$1$ cases.

\begin{theorem}\label{t1.1}
	Consider the system (\ref{0.1}) of codimension $1$ satisfying {\bf (A1)(A2)}. It is globally controllable if and only if
	\begin{align}\label{1.4}
		0\in\Int\Con\{G_*f(x),G|_x\},\quad\forall x\in\R^n,
	\end{align}
	where $\Con$ is the convex hull of vectors in $T_xM$.
	
\end{theorem}

\begin{proof}
	The sufficiency was proved in \cite[Theorem 6.1]{che10}. Here we show the necessity.
	
	Suppose the condition (\ref{1.4}) is not satisfied, then $\forall x\in \R^n$ there exists a non-vanishing $G$-invariant smooth vector field $\xi$ on $\I_G(x)$ such that $\forall x\in\R^n$, $f|_x\in G^*|_x$, where $G^*|_x$ is the complement in $T_x\R^n$ of the connected component of $T_x\R^n\backslash G|_x$ containing $\xi|_x$; which is denoted as $f\in G^*$ for short.
	
	Since according to Lemma \ref{l1.2} the whole space is separated by $\I_G(x)$, and $\xi$ as a normal vector field actually gives an orientation of $\I_G(x)$, one can see that a trajectory of (\ref{0.1}) starting from $\I_G(x)$ cannot cross it in positive time, since $f\in G^*$ means that the direction of $f$ on $\I_G(x)$ is either zero or equal to the orientation normal vector field. So the system is not globally controllable.
\end{proof}

\begin{remark}\label{rem}
	The condition (\ref{1.4}) aims to separate $\A_f(x)$ and $\A_{-f}(x)$, the latter of which is the negative-time reachable set of (\ref{0.1}). 
	{By the Baker-Campell-Hausdorff formula, (\ref{1.4}) is actually equivalent to 
		\begin{align}\label{1.4a}
			0\in\Int\Con\{G|_x,\ad_Gf|_x\},~\forall x\in\R^n.
	\end{align}}
\end{remark}
Due to topological reasons, we cannot always find $n-1$ non-vanishing smooth vector fields that pointwise generate a codimension-$1$ distribution. But when this is achievable, the above criterion (\ref{1.4}) can be simplified to a form that is easier to verify.

\begin{corollary}\label{co1}
	Consider a codimension-$1$ system {satisfying {\bf (A1)(A2)}}. If there exist $n-1$ vector fields $\tilde{g}_1,\cdots,\tilde{g}_{n-1}$ such that $\forall x\in \R^n$, $G|_x=\Span_{\R}\{\tilde{g}_1|_x,\cdots,\tilde{g}_{n-1}|_x\}$, construct the following criterion function
	\begin{align}\label{1.5}
		C(x):=\det(f|_x,\tilde{g}_1|_x,\cdots,\tilde{g}_{n-1}|_x),
	\end{align}
	then the system is globally controllable if and only if for $\forall x\in \R^n$, $\exists y_1,y_2\in \I_G(x)$, s.t. $C(y_1)C(y_2)<0$, i.e. the criterion function (\ref{1.5}) changes its sign on each leaf.
\end{corollary}

The proof is straightforward by noticing that (\ref{1.5}) aims to verify whether $f$ gives an orientation of the submanifold $\I_G(x)$. This corollary this
corollary confirms partly the conjectures in \cite{su16}.

\section{Conditions for Controllability: Codimension $k>1$}\label{S3}

Now consider systems of higher codimensions. As an analogue of {\bf (A2)}, to ensure that integral manifolds are embeddings, we make the following assumption on the system (\ref{0.1}).
\begin{enumerate}
	\item[\bf (A2')] $\forall p\in \R^n$, given any $(n-k+1)$-dimensional cube in $\R^n$, $\I_G(p)$ is not dense in it {with respect to the $(n-k+1)$-dimensional Euclidean topology}.
\end{enumerate}

\begin{remark}
	In fact, one may change {\bf (A2')} to a weaker form: if for $\forall x$, $\I_G(x)$ is of finite depth \cite{ni77}, then each of them is an embedding (while the converse is not always true).
	%Given a foliation and denote the set of its leaves by $F$, then by finite depth, we mean that
	%		{\small $$
		%			\underset{L\in F}{\sup}\sup\{k~|~\exists L_1,\cdots,L_k\in F,~\text{s.t.}~L_1<\cdots<L_k=L\}<\infty,
		%			$$}
	%		where the relation $<$ between leaves is defined as follows: $\forall L_1,L_2\in F$, $L_1<L_2\La$ $L_1\in \overline{L_2}$ .
\end{remark}

We will show that Theorem \ref{t1.1} still holds for codimension-$2$ systems under the assumption {\bf (A2')}, and the condition (\ref{1.4}) is necessary and sufficient for global controllability of codimension $k$ ($k>2$) systems under certain restrictions.

First, we define the supporting distribution of vector fields.

\begin{definition}
	Consider a codimension-$k$ distribution $D$ on a manifold $M$, denote its complement in $\X(M)$ by $D^{\perp}$, i.e. $D|_x\oplus D^{\perp}|_x=T_xM$, $\forall x\in M$. Given a $(k-1)$-dimensional distribution $S\subset D^{\perp}$, a set $F\subset \X(M)$ is called $S$-supported along $D$ if $S$ is $D$-invariant and there exists a $D$-invariant non-vanishing smooth vector field $\xi\in H\backslash S$, such that $\forall x\in U$, $F|_x\subset S^+|_x$, where $S^+|_x$ is the complement in $D^{\perp}$ of the connected component of $(D^{\perp}\backslash S)|_x$ containing $\xi|_x$, and this is briefly denoted as $f\in S^+$.
\end{definition}

{Since $\Con\{G_*f(x),G|_x\}$ is an affine subspace in $T_xM$,} one can see that  under the assumption {\bf (A1)} any system (\ref{0.1}) which does not satisfy (\ref{1.4}) will allow a supporting distribution $S$ along $G$. Further, we have the following statement on the necessity of (\ref{1.4}) for a class of codimension-$k$ systems.

\begin{theorem}\label{t1.2}
	Consider the system (\ref{0.1}) of codimension $k$ satisfying {\bf (A1)(A2')}. If the condition (\ref{1.4}) is not satisfied and for the supporting distribution $S$ we have
	\begin{align}\label{1.6}
		f|_x\notin\Lie(S)|_x,~\forall x\in \R^n,
	\end{align}
	then the system is not globally controllable.
\end{theorem}
\begin{proof}
	If $f\notin\Lie(S)$, then $G\oplus S$ is a codimension-$1$ involutive distribution; therefore each of its maximal integral manifolds will separate $\R^n$ into two components. By similar arguments as in the proof of Theorem \ref{t1.1}, since $f$ is $S$-supported along $G$, a trajectory starting from $x\in\R^n$ will remain in one connected component of $\R^n\backslash\I_{G\oplus S}(x)$ and its boundary, preventing the system from being controllable.
\end{proof}

Since $1$-dimensional distributions are always involutive, we have the following corollary immediately.

\begin{corollary}\label{c1.2}
	For codimension-$2$ systems satisfying {\bf (A1)(A2')}, the condition (\ref{1.4}) is necessary and sufficient for global controllability.
\end{corollary}

{We give an example of codimension 3 to illustrate the necessity of the condition.}

\begin{example}
	Consider a system (\ref{0.1}) in $\R^6$, where $m=2$, and
	\begin{align*}
		g_1(x_1,\cdots,x_6)&=(1,0,0,0,0,0)^{\T},\\
		g_2(x_1,\cdots,x_6)&=(0,1,x_1,0,0,0)^{\T},\\
		f(x_1,\cdots,x_6)&=(0,0,0,1,x_1,x_1^2)^{\T},
	\end{align*}
	one can check that it satisfies {\bf (A1)(A2)}. To be specific, $G=\Lie\{g_1,g_2\}=\Span_{\R}\{(1,0,0,0,0,0)^{\T},$ $(0,1,0,0,0,0)^{\T},$ $(0,0,1,0,0,0)^{\T}\}$, $\dim G=3$; $\Span_{\R}\{f,\Lie\{g_1,g_2,f\}\}=\R^6$. That is to say, this system is strongly attainable; however, choosing supporting distribution as $S:=(0,0,1,1,0,0)^{\T}$, it follows that the system satisfies (\ref{1.6}) but does not satisfy (\ref{1.4}) by Remark \ref{rem}. We can check that $(0,0,0,-1,0,0)^{\T}$ does not belong to the reachable set of $(0,0,0,0,0,0)$, since $(0,0,0,0,0,0)$ is not in the interior of $\Con\{g_1,g_2,\ad_G{f}\}$.
\end{example}

%x	Further, when the supporting distribution is involutive we claim that there exists a partial order on the set of the leaves of the foliation.

\begin{remark}\label{p1.1}
	Consider the control-affine system (\ref{0.1}) of codimension $k$ satisfying {\bf (A1)(A2)}. If (\ref{1.4}) is not satisfied and the corresponding supporting distribution $S$ of $f$ is involutive, then $\forall x\in \R^n$, there exists a neighbourhood of $\I_G(x)$ such that all the leaves in this neighbourhood are partially ordered by the flow of $f$. However, if the supporting distribution Lie-generates $f$, this local partial order on the foliation cannot be properly defined, since in that case, $f$ may be recurrent \cite{bo19}, allowing a point to move from leaves of lower order to those of higher order. 
\end{remark}

Note that the topological structure of a manifold tangent to a distribution $S$ may be complicated when $S$ is not involutive; \cite{ba} has shown that the Hausdorff dimension of a $k$-dimensional characteristic submanifold tangent to a $k$-dimensional non-involutive distribution is less than $k-1$, so, in that case, the product manifold will certainly not separate the whole state space.

%	From the above theorem we may derive another restriction on $f$ to make the condition (\ref{1.4}) necessary and sufficient for codimension-$k$ cases.
%	
%	\begin{theorem}\label{t1.3}
	%		Consider system (\ref{0.1}). If $\forall x\in \R^n$, There exits a neighbourhood of  is 
	%	\end{theorem}

Finally, as the sufficiency of the condition (\ref{1.4}) for global controllability has been proved for codimension-$k$ systems \cite{che10}, we conjecture that the condition in theorem \ref{t1.2} is the one needs to make (\ref{1.4}) necessary and sufficient.

\begin{conjecture}
	Consider (\ref{0.1}) satisfying {\bf (A1)(A2')}. If (\ref{1.4}) is not satisfied and for the corresponding  supporting distribution $S$ one has $f\in \Lie(S)$, then the system is globally controllable.
\end{conjecture}

The argument above can be applied analogously to switched systems, making the conditions in \cite{che06} necessary and sufficient. We state the conclusion as follows.
\begin{theorem}
	Consider a switched control-affine system
	\begin{align}\label{4.3}
		\dot{x}(t)=f^{\sigma(t)}|_{x(t)}+\sum_{i=1}^{m}u^i(t)g_i^{\sigma(t)}|_{x(t)}
	\end{align}
	where $\sigma:[0,+\infty\}\ra\{1,\cdots,N\}$ is a measurable right-continuous mapping called the switching signal, and $\{f^j\}_{j=1}^{N}$, $\{g_i^j\}_{i=1,\cdots,m}^{j=1,\cdots,N}$ are smooth vector fields on $\R^n$. Let $G:=\Lie\{g_i^j\}_{i=1,\cdots,m}^{j=1,\cdots,N}$, $F:=\{f^j\}_{j=1}^{N}$, then if $G$ satisfies assumptions {\bf (A1)(A2')} and 
	\begin{align}\label{4.4}
		0\in\Int\Con\{G_*F(x),G|_x\},\quad\forall x\in\R^n,
	\end{align}
	then (\ref{4.4}) is necessary and sufficient for global controllability of system (\ref{4.3}) of codimension-$1$ and codimension-$2$, and when the system (\ref{4.3}) is regular of codimension $k>2$ satisfying
	\begin{align}\label{4.5}
		F|_x\nsubseteq\Lie(S)|_x,~\forall x\in \R^n,
	\end{align}
	the system satisfies (\ref{4.4}) if it is globally controllable.
\end{theorem}

\section{Motion Planning of Drifted Control-Affine Systems}

In this section, we consider motion planning problems of the system (\ref{0.1}). For a class of codimension-$1$ systems, we directly lift curves in state spaces to control spaces in two steps; as for general affine systems, we extend the HCM to some special kinds of strong bracket-generating drifted affine systems.

{From now on, assume the manifold $M$ on which the control system (\ref{0.1}) is defined to be diffeomorphic to $\R^n$ with Euclidean topology.}

\subsection{Simplest Case of Codimension $1$}

The simplest case is as in Corollary \ref{co1}: the distribution is globally generated by $n-1$ non-vanishing vector fields, and we assume that the points where $f\in G$ form a connected set of zero measure. Given $x\in M$ and an orientation of $\I_G(x)$, the whole space is then divided into three parts: the set above $\I_G(x)$, the set beneath $\I_G(x)$ with $f$ along the positive orientation, and the set beneath $\I_G(x)$ with $f$ along the negative orientation.

When the starting and ending points are on different sides of the separating plane, we design the algorithm in several steps: first, drive the trajectory along the direction where $C(x)$ diminishes; when the orientation of the drift vector field is reversed, then for any two points $p,q$ lying in a partially ordered foliated space as described in Remark \ref{p1.1} and for any curve $\gamma(t)$ connecting $p$ and $q$ flowing from higher order slices to lower order ones, that is to say, one may lift it to an admissible trajectory of \ref{0.1}.

Taking $\langle\cdot,\cdot\rangle$ as conventional Euclidean inner product, a curve $\gamma:[0,T]\ra M$ is called along a vector field $f$ if it satisfies $\langle \dot{\gamma}(t),f|_{\gamma(t)}\rangle>0$, $\forall t>0$. We restate the above algorithm driving the trajectory from a point $x$ above $\I_G(x)$ to a point $y$ below it. Assume the controls are bounded by $K$, and the sampling time is $T$.
\begin{enumerate}
	\item Step 1. Find a point $p\in \{q|C(q)=0\}$, draw a curve $\gamma_1:[0,1]\ra M$ along $f$ connecting $x$ and $p$;
	\item Step 2. Draw a curve $\gamma_2:[0,1]\ra M$ along $f$ connecting $p$ and $y$.
\end{enumerate}
The concatenation of $\gamma_1$, $\gamma_2$ denoted by $\gamma$, will be an admissible curve of (\ref{0.1}) and giving the control value at each point as the coefficients of the linear span of $\dot{\gamma}(t)-f|_{\gamma(t)}$ with respect to $\{g_i|_{\gamma(t)}\}_{i=1}^m$. {In application, one may choose the curves $\gamma(t)$ as bounded, and there is a tradeoff between the bound of controls and the reaching time.}
\begin{example}\label{ex2}
	Consider a system (\ref{0.1}) with $m=3$ in $\R^4$, let the coordinate be $(x_1,x_2,x_3,x_4)$, the controls lie in $[-2,2]$, and 
	\begin{align*}
		g_1&=(1,0,0,0), ~g_2=(0,1,x_2,0),\\
		g_3&=(0,0,1,0), ~f=(0,0,1,x_3).
	\end{align*}
	One can check that the zeros of the criterion function (\ref{1.5}) is $H:=\{(x_1,x_2,x_3,x_4)\in \R^4|x_3=0\}$. If we want to steer the system from $x=(0,0,1,1)$ to $y=(0,0,-2,-1)$, which are points lying on different sides of the hypersurface spanned by $g_1,g_2,g_3$, then using the two-step algorithm above, we may first let $u(t)=(0,0,-2)$ to drive $x$ to $p=(0,0,0,\frac{3}{2})\in H$ during time $t\in[0,1]$, and let $u(t)=(0,0,-\frac{9}{5})$ drive $p$ to $y$ during time $t\in[1,\frac{7}{2}]$.
	
\end{example}

\subsection{HCM For Drifted Affine Systems}
If the system (\ref{0.1}) is globally controllable, while the distribution $G$ is regular but not globally generated, we use the HCM to compute the steering control, since the solution to the HLE exists globally on the domain.

We will show that the argument in \cite{ch06}\cite{su93} still holds for a class of drifted systems.

In this section we make the following assumptions: 
\begin{enumerate}
	\item[\bf (A3)]The driftless part of the system (\ref{0.1}) is strong bracket-generating \cite{ch06}, i.e. $\forall\eta\in\R^m\backslash\{0\}$, $\forall x\in \R^n$, denote $\eta\cdot g:=\sum_{i=1}^m\eta_ig_i$,
	$$
	\Span\Big\{\{g_i|_x\}_{i=1}^m,\{[\eta\cdot g,g_i]|_x\}_{i=1}^m\Big\}=T_xM;
	$$
	\item[\bf (A4)] The drift vector field $f$ satisfies $$[f,g_i]|_x\in \Span\{g_i|_x\}_{i=1}^m, ~i=1,\cdots,m.$$
\end{enumerate}

We define the endpoint map of system (\ref{0.1}) with respect to $p\in M$ as
$\E_p:(0,+\infty)\times{\cal U}\ra M$, $(T,u(t))\mapsto \gamma_{u,p}(T)$, where $\gamma_{u,p}:[0,T]\ra M$ is the trajectory of the system (\ref{0.1}) starting at $p$ and corresponding to $u(t)$.

The tangent of the exponential map is 
$$
\begin{array}{rl}
	~&D\E_p|_{T,u}(\tau,v)\\
	=&\int_{0}^{1}\tau(P^{T,u}_{t,1})_*f|_{\gamma_u(t)}+\sum_{i=1}^{m}v^i(t)(P^{T,u}_{t,1})_*g_i|_{\gamma_u(t)}\dd t
\end{array}
$$
where $P^u_{t,1}$ is the diffeomorphism generated by (\ref{0.1}) that maps a point $p$ to $\gamma_{u,p}(1)$, where $\gamma_{u,p}$ is the solution to (\ref{0.1}) that satisfies $\gamma(t)=p$; $(P^{T,u}_{t,1})_*$ denotes its tangent map. A control that degenerates $D\E_p$ is called abnormal, and its image under $\E|_p$ is called the singular set.

The adjoint map of $D\E_p$ is given by 
\begin{align*}
	(D\E_p|_{T,u})^*:T^*M&\ra (0,+\infty)\times{\cal U}\\
	z&\mapsto(\tilde{T},\vphi_{1,u,z},\cdots,\vphi_{m,u,z})
\end{align*}
where $z\in T^*M$, $\vphi_{i,u,z}(t):=\langle\lambda_z(t),g_i|_{\gamma_{u,p}(t)}\rangle$, $i=1,\cdots,m$ are called the switching functions, with $(\gamma_{u,p},\lambda_z):[0,T]\ra T^*M$ being the solution of the Hamiltonian system satisfying $\gamma_{u,p}(0)=p$, $\lambda_z(T)=z$ with respect to the following Hamiltonian:
\begin{align}\label{5.2}
	H_u(x,z):=\langle z,f|_x\rangle+\sum_{i=1}^mu^i\langle z,g_i|_x\rangle-\frac{\|u\|^2}{2},
\end{align}
that is to say, $\gamma_{u,p}(t)$ is a controlled trajectory of (\ref{0.1}) and $\lambda_z$ solves the following adjoint equation with terminal value $z$:
\begin{align}\label{5.4}
	\dot{\lambda}(t)=-\lambda(t)\big(Df|_{\gamma_{u,p}(t)}+\sum_{i=1}^mu^i(t)Dg_i|_{\gamma_{u,p}(t)}\big).
\end{align}

From now on, $\|\cdot\|$ stands for the Euclidean norm; denote by $\vphi(t)=(\vphi_{1,u,z}(t),\cdots,\vphi_{m,u,z}(t))$ when there is no misunderstanding on $u$ and $z$ corresponding to the switching function.

Define $b_i(t):=\langle\lambda_z(t),[f,g_i]|_{\gamma_{u,p}(t)}\rangle$, $i=1,\cdots,m$, where $(\gamma_{u,p},\lambda_z)$ is the solution to the above Hamitonian system. Since $[f,g_i]\in \Span\{g_i\}_{i=1,\cdots,m}$, denoting $b(t):=(b_1(t),\cdots,b_m(t))$, we have $b(t)=E(t)\vphi(t)$, where $E(t)$ is an $n$-dimensional time-varying square matrix. Define $B(t):=\int_0^tb(s)\dd s$. When $(\gamma(t),\lambda(t))$ lies on a compact set, obviously $\exists E>0$, s.t. $\|B(t)\|\leqslant E\int_0^t\|\vphi(s)\|\dd s$.

In the drifted case under assumption {\bf (A4)},  the switching function satisfies
\begin{align}\label{5.3}
	&\dot{\vphi}(t)=u(t)\psi(t)+b(t)=u(t)\psi(t)+E(t)\vphi(t),\\\nonumber
	&\psi(t):=(\psi_{ij}(t))_{m\times m}, \quad\psi_{ij}(t):=\langle \lambda_z(t),[g_i,g_j]_{\gamma_{u,p}(t)}\rangle,
\end{align}
with $\gamma_{u,p}$, $\lambda_z$ defined as before in (\ref{5.2})(\ref{5.4}). Then under assumption {\bf (A3)}, when $\dot{\vphi}=0$, $\psi$ is nonsingular, hence the only abnormal control is $u\equiv0$, and the singular set is hence $\{\vphi^f_t(p)\}_{t>0}$.

Similar to the case of \cite{su93}, to prove the global existence on $[0,1]$ of the solution to the PLE (\ref{0.3}) for system (\ref{0.1}), we only need to show the norm of the adjoint map of $D\E_p|_{T,u}$ is of linear growth with respect to $\|(T,u)\|_{L^2}$, that is to say, for any compact set $K\subset M$, $\exists c>0$, s.t.
$$
\sum_{i=1}^m\int_0^1(\vphi_{i}(t))^2+T^2\dd t\geqslant\frac{c\|z\|^2}{1+\|(T,u)\|_{L^2}^2}
$$
holds for all $u\in \cal U$ satisfying $\E_p(u)\in K$, $\forall z\in T^*_{\E_p(u)}M$.

\begin{theorem}\label{t5.1}
	Consider the system (\ref{0.1}) satisfying (\ref{1.4}), assuming ${\cal U}=L^2([0,T],\R^m)$. If (\ref{0.1}) is SBG and the singular set is $\I_f^+$, then for any compact subset $K\subset M\backslash\{\vphi^f_t(p)\}_{t\geqs0}$, $\exists c_K>0$ s.t. $\forall u\in \E_p^{-1}(K)$, $\forall z\in T^*_{\E_p(u)}M$ with $\|z\|=1$, we have
	$$
	\int_{0}^1\|u(t)\|\dd t\int_0^1\|\vphi(t)\|\dd t>c_K,
	$$
	where $\vphi(t)=(\vphi_{1,u,z}(t),\cdots,\vphi_{m,u,z}(t))$, and $\lambda(t)$ is the solution ending at $z$ of the adjoint equation (\ref{5.4}) along $\gamma_{u,p}(t)$, $u(t)$.
\end{theorem}
\begin{proof}
	For $\alpha>0~\text{s.t.}~d(p,K)>\alpha$, define two compact sets 
	\begin{align*}
		K^*&=\{(x,z)\in T^*M|x\in K,\|z\|=1\},\\
		K^*_{\alpha}&=\{(x,z)\in T^*M|d(x,K)<\alpha,\frac{1}{2}<\|z\|<2\},
	\end{align*}
	such that $K^*\cap (T^*M\backslash K^*_{\alpha})=\varnothing$.
	Construct a function $\theta:T^*M\ra\R$ satisfying $\theta|_{K^*}\equiv0$, $\theta|_{T^*M\backslash K^*_{\alpha}}>1$, and denote in brief by $\theta(t)$ the composition $\theta(\gamma(t),\lambda(t))$. Obviously $(\gamma(0),\lambda(0))\in T^*M\backslash K^*_{\alpha}$ and $(\gamma(T),\lambda(T))\in K^*$. Then $\exists s\in(0,1)$ s.t. $\theta(s)=1$. Since $\theta(1)=0$, it follows that
	$$
	\Big|\int_s^1\dot{\theta}(t)\dd t\Big|=\Big|\int_s^1\langle u(t),\mu(t)\rangle\dd t\Big|>1,
	$$
	where $\langle u(t),\mu(t)\rangle:=\sum_{i=1}^mu^i(t)\mu_i(t)$, $\mu_i(t):=L_{f^*_i}\theta(t)$ ($L_{g_i^*}$ is the directional derivative of a function along the vector field $g_i^*\in\X(T^*M)$ which is the dual vector field of $g_i$ with respect to the canonical symplectic form), $i=1,\cdots,m$.
	
	Choose a positive constant $\rho>0$. If $\|\vphi(t)\|>\rho$ on $[s,1]$, then $\exists C_{\mu}>0$, s.t.
	$$
	\begin{array}{rl}
		\int_s^1\langle u(t),\mu(t)\rangle\dd t&<\int_s^1\frac{\|\vphi(t)\|}{\rho}\langle u(t),\mu(t)\rangle\dd  t\\
		~&<\frac{C_{\mu}}{\rho}\int_s^1\|\vphi(t)\|\dd t\int_s^1\|u(t)\|\dd t,
	\end{array}
	$$
	and hence 
	$$
	\int_{0}^1\|u(t)\|\dd t\int_0^1\|\vphi(t)\|\dd t>\frac{\rho}{C_{\mu}}>0.
	$$
	
	Next, we consider the case when $\|\vphi(t)\|<\rho$ on $[s,1]$.
	
	Define 
	$$
	A(x,t):=\frac{1}{\det(\psi(t))}\langle x\psi^+(t),\mu(t)\rangle, ~x\in \R^m,
	$$ where $\psi^+$ is the matrix of the complementary minors of $\psi$. Then by (\ref{5.3}), integrating by parts yields
	$$
	\begin{array}{rl}
		~&\int_s^1\langle u(t),\mu(t)\rangle\dd t\\
		=&\int_s^1A(\dot{\vphi}(t)-b(t),t)\dd t\\
		=&A(\vphi(1)-B(1),1)-A(\vphi(s)-B(s),s)\\
		~&-\int_s^1D_2A(\vphi(t)-B(t),t)\dd t
	\end{array}
	$$
	where $D_2A$ is the partial derivative of $A$ with respect to the second variable. Since $A(v,t)$ is bounded by $C_1\|v\|$ where $C_1>0$, and when $\|\vphi\|\leqs\rho$, $\|\vphi(t)-B(t)\|$ is bounded by $C_2\rho$, we have 
	$$
	A(\vphi(1)-B(1),1)-A(\vphi(s)-B(s),s)\leqs C_3\rho=2C_1C_2\rho,
	$$
	and $D_2A(v,t)\leqs C_4\|\vphi(t)-B(t)\|\|u(t)\|$, hence
	$$
	\begin{array}{rl}
		~&\int_s^1D_2A(\vphi(t)-B(t),t)\dd t\\
		\leqs& C_4(\int_s^1\|\vphi(t)\|\|u(t)\|\dd t+\int_0^1E\|\vphi(t)\|\dd t\int_s^1\|u(t)\|\dd t)\\
		\leqs&C_4\int_0^1\|\vphi(t)\|\dd t\int_0^1\|u(t)\|\dd t.
	\end{array}
	$$
	Therefore we have 
	$C_4\int_0^1\|\vphi(t)\|\dd t\int_0^1\|u(t)\|\dd t+C_3\rho\geqslant1$.
	Adjust $\rho$ such that $C_3\rho\leqs\frac{1}{2}$, it follows that 
	$$
	\int_{0}^1\|u(t)\|\dd t\int_0^1\|\vphi(t)\|\dd t>\frac{1}{C_4},
	$$
	and $C_K=\min\{\frac{\rho}{C_{\mu}},\frac{1}{C_4}\}$ is the lower bound we desired.
	
	In the above discussion, we assume conditions on the norm of $\vphi$ on the whole interval; obviously the argument still holds if $\|\vphi(t)\|\geqs\rho$ ($\leqs\rho$) on some subinterval.
\end{proof}

Theorem \ref{t5.1} shows that any nonsingular trajectory on a compact subset of the state space can be lifted to a trajectory in the control space by solving the PLE (\ref{0.3}), giving the steering control as the final value of the solution.

\section{Conclusion}

In this paper, we studied the conditions for global controllability of affine systems with drifts. Under the assumptions of certain regularities, necessary and sufficient conditions have been proved for codimension $1$ and $2$ systems, and necessary conditions are given for specific systems of higher codimensions. Then, the motion planning problems for globally controllable drifting affine systems are investigated, improving the applicability of the HCM method.

The key point of the controllability condition is the separation property of foliations. The discussions can all be extended to systems on simply connected manifolds since Lemma \ref{l1.2} holds for any (edgeless) manifold $M$ satisfying $H_1(M,\Z/2\Z)=0$. Further, if $M$ is compact, {\bf (A2)} can be replaced by that each leaf $\I_G(x)$ is of finite depth since this is equivalent to that each $\I_G(x)$ is an embedded submanifold.

\bibliographystyle{IEEEtran}

\end{document}